\documentclass[12pt]{amsart}
\usepackage{amsmath,amsthm, amssymb}
\usepackage{graphicx,epic}
\usepackage{mathrsfs}
\usepackage{hyperref}
\usepackage{float}
\usepackage{color}
\usepackage{xcolor}

\newtheorem{theorem}{Theorem}
\newtheorem{proposition}{Proposition}[section]

\newtheorem{lemma}[proposition]{Lemma}

\newtheorem{definition}[proposition]{Definition}

\newtheorem*{Acknowledgements}{Acknowledgements}

\def\ie{{\em i.e.,\ }}
\def\eg{{\em e.g.\ }}
\def\eps{\varepsilon}

\def\R{{\mathbb R}}

\newcommand {\CC}{{\mathcal C}}

\newcommand {\CE}{{\mathcal E}}

\newcommand {\CP}{{\mathcal P}}

\def\Leb{{\hbox{{\rm Leb}}}}

\def\s{\sigma}

\def\1{ {\hbox{{\it 1}} \!\! I} }

\def\al{\alpha}

\def\La{\Lambda}
\def\la{\lambda}

\def\8{\infty}

\def\disp{}

\usepackage{float,color}

\newcommand{\ninf}{{n\to+\8}}
\newcommand{\mv}{{\mathbf{v}}}
\newcommand{\mw}{{\mathbf{w}}}

\theoremstyle{definition}
\newtheorem{remark}{Remark}

\begin{document}
\synctex =1
\title[SRB measures]
{SRB measures for higher dimensional singular partially hyperbolic attractors}
\author{Renaud Leplaideur and Dawei Yang}
\date{Version of \today}
\thanks{}

\thanks{RL would like to thank Soochow University for its support and hospitality. DY would like to thank the support of NSFC 11271152 and A Project Funded by the Priority Academic Program Development of Jiangsu Higher Education Institutions(PAPD)}

\maketitle

\begin{abstract}
We prove the existence and the uniqueness of the SRB measure for any singular  hyperbolic attractor in dimension $d\ge 3$. The proof does not use Poincar\'e sectional maps, but uses basic properties of thermodynamical formalism. 
\end{abstract}

\section{Introduction}\label{sec:intro}

\subsection{Background}

Our aim is to prove that any finite dimensional singular partially hyperbolic attractor admits a unique SRB measure.

SRB stands for Sinai, Ruelle and Bowen who established  in the 70's the theory of  thermodynamical formalism and proved the existence and the uniqueness of some special measures for Anosov systems and Axiom A attractors (diffeomorphisms or flows), see \eg \cite{Sin72,BoR75,Rue76}. 


SRB measures usually satisfy the two following properties. On the one hand their disintegrations along unstable manifolds are absolutely continuous with respect to the Lebesgue measures on these manifolds. On the other hand, and at least for the hyperbolic case,  SRB measures are expected to be {\em physical measures}, i.e.,  their sets of generic points have positive Lebesgue measure on the manifold (see \eg \cite{young-srb} for a survey).

Since Sinai, Ruelle and Bowen, it has remained a challenge to export the SRB theory to other dynamical systems presenting a weaker form of hyperbolicity. Among them, the Lorenz-like attractors represent one of the most famous families. First introduced by Lorenz in \cite{Lorenz}, this class has several typical properties of chaotic dynamics: it is robust in the $C^{1}$-topology, every ergodic invariant measure is hyperbolic (see Lemma \ref{Lem:measure-hyperbolic}) but  the attractors themselves are not  hyperbolic.

For the 3-dimensional case,  Guckenheimer  {was inspired by} the example introduced by Lorenz to define  in \cite{Guc76} the notion of geometric Lorenz attractor. He asked about the existence of SRB measures for these attractors. In \cite{MPP04}, 
 Morales Pacifico and Pujals  used the notion of \emph{singular hyperbolicity} to characterize more general Lorenz-like dynamics.  
 In \cite{A-V09}, Araujo-Pacifico-Pujals-Viana proved the existence and the uniqueness of SRB measures for \emph{three-dimensional singular hyperbolic attractors}. We also mention other related works:\cite{Col02,ArP07}.

 Higher-dimensional singular hyperbolic attractors have been defined in \cite{ZGW08} and \cite{MeM08}. It was thus a natural question to investigate the existence and the uniqueness of the SRB measure for these attractors. In this paper we actually prove  the existence and the uniqueness of such a measure.

We emphasize that the method from \cite{A-V09} cannot be adapted to the higher dimensional case. Indeed, their main idea is to construct a sequence of cross-sections and consider the return maps among these sections. Then, they quotient the dynamics to piecewise expanding maps of intervals and use the existence of absolutely continuous invariant measures for these maps. That idea cannot be directly adapted to the higher-dimensional case for the following two reasons. On the one hand, it is much more difficult to construct cross sections,  and on the other hand, the existence of the absolutely continuous invariant measure for general higher-dimensional piecewise expanding maps may fail (see \cite{Tsujii, Buzzi}).

Our strategy is different. We show here that basic properties of thermodynamical formalism allow one to show the existence and the uniqueness of SRB measures. Actually, a close strategy has already been investigated in \cite{CoY05} where Cowieson and Young obtained SRB measures by using the Pesin entropy formula for diffeomorphisms beyond uniform hyperbolicity. 

After we had posted the paper on ArXiV, we had some communication with M. Viana, and learnt that he and J. Yang had a work in progress on that topic. We were not aware on that work, which did not exist as a preprint but as a recorded talk (see \cite{Yan-video}). We agreed with them to mention this video. In our mind, the fact that we independently got similar ideas emphasizes how natural this strategy is.

Finally, we want to emphasize that our strategy deeply uses the absolute continuity of the strong stable foliation, which has been a very important research tool for years (see \cite{Pes04}).

\subsection{Settings and the statement of the result}

Let $X$ be a vector field on a $d$-dimensional manifold $M$, and $\varphi_t$ be the flow generated by $X$. We recall that a compact invariant set $\Lambda$ is called a topological attractor if
\begin{itemize}

\item there is an open neighborhood $U$ of $\Lambda$ such that $\cap_{t\ge 0}\varphi_t({\overline U})=\Lambda$,

\item $\Lambda$ is transitive, i.e., there is a point $x\in\Lambda$ with a dense forward-orbit. 

\end{itemize}

A compact invariant set $\Lambda$ is called a \emph{singular hyperbolic attractor} (see  \cite{MeM08,ZGW08}) if it is a topological attractor, with at least one \emph{singularity} $\sigma$, which means that $\sigma$ satisfies  $X(\sigma)=0$.
Moreover, there is a continuous invariant splitting $T_\Lambda M=E^{ss}\oplus F^{cu}$ of ${\rm D}\varphi_t$ together with constants $C>0$ and $\lambda>0$ such that
\begin{itemize}


\item Domination: for any $x\in\Lambda$ and any $t>0$, $\|{\rm D}\varphi_t|_{E^{ss}(x)}\|\|{\rm D\varphi_{-t}}|_{F^{cu}(\varphi_t(x))}\|\le C{\rm e}^{-\lambda t}$.

\item Contraction: for any $x\in\Lambda$ and any $t>0$, $\|{\rm D}\varphi_t|_{E^{ss}(x)}\|\le C{\rm e}^{-\lambda t}$.

\item Sectional expansion: for any $x$, {$F^{cu}(x)$ contains two non-collinear vectors}, and any $t>0$, for every pair of non-collinear vectors $\mv$ and $\mw$ in $F^{cu}(x)$, 
$|\det {\rm D}{\varphi_{t}}|_{{\rm span}<\mv,\mw>}|\ge C{\rm e}^{\lambda{t}}.$

\end{itemize}

We  emphasize that one of the difficulties to study these attractors is that  the singularity may belong  to the attractor  and may be accumulated by recurrent regular orbits. {Since the uniformly hyperbolic case has already been well-understood since \cite{BoR75}, we assume that the attractor does contain a singularity.}

We recall that entropy for a flow is defined as being the entropy of the time-1 map $f:=\varphi_{1}$. 

\begin{definition}\label{def-srb}
An invariant measure for the flow is said to be an SRB measure if it is an SRB measure for the time-one map $f$ of the flow, that is:
\begin{enumerate}
\item it has a positive Lyapunov exponent almost everywhere,
\item  the conditional measures on unstable manifolds are absolutely continuous w.r.t. Lebesgue measures on unstable manifolds.
\end{enumerate}
\end{definition}
We refer the reader to the survey \cite{young-srb}  on SRB measures for more details. Ledrappier-Young \cite{LeY85} proved that $\mu$ is an SRB measure if and only if $\mu$ has a positive Lyapunov exponent almost everywhere and the entropy of $\mu$ is the integration of the sum of all positive Lyapunov exponents.

The equality is called \emph{the Pesin entropy formula} and is the heart of our strategy. 

The goal of this paper is to prove the following theorem.

\begin{theorem}\label{th-main}
Let $\Lambda$ be a singular hyperbolic attractor of a {$\CC^{2}$} vector field $X$ of \emph{any} manifold $M$ {of any dimension}. Then there is a unique SRB measure $\mu$ supported on $\Lambda$. Moreover
$\mu$ is ergodic and physical.
\end{theorem}

\subsection{Thermodynamical formalism}\label{Sec:thermo}

The main idea of the proof of the theorem is to use the thermodynamical formalism. We refer the reader to \cite{Wal82} for general results on this topic. 

We recall that if $\psi:\Lambda\to\R$ is a continuous function, an invariant probability measure $\mu$ is said to be \emph{an equilibrium state for $\psi$} if it satisfies 
$$h_{\mu}(f)+\int\psi\,d\mu=\sup_{\nu~\textrm{ergodic},~{\rm supp}\nu\subset\Lambda}\left\{h_{\nu}(f)+\int\psi\,d\nu\right\}=:{\CP(\psi)}.$$
 $\CP(\psi)$ is called the \emph{pressure} for the potential $\psi$. 
We recall that it also satisfies
\begin{eqnarray*}
\hskip -1cm\CP(\psi)&=&\lim_{\eps\to0}\limsup_{\ninf}\frac1n\sup\left\{\sum_{x\in \CE_{n,\eps}}e^{S_{n}(\psi)(x)},\ \CE_{n,\eps}=\textrm{maximal~}(n,\eps)\text{-separated set in }\La\right\}\\
&=&\lim_{\eps\to0}\limsup_{\ninf}\frac1n\inf\left\{\sum_{x\in \CE_{n,\eps}}e^{S_{n}(\psi)(x)},\ \CE_{n,\eps}=\textrm{minimal~}(n,\eps)\text{-spanning set in }\La\right\},
\end{eqnarray*}
where $S_{n}(\psi)=\psi+\ldots +\psi\circ f^{n-1}$. Recall that $x,y\in M$ are $(n,\eps)$-close if $d(f^i(x),f^i(y))<\varepsilon$ for any $0\le i\le n-1$; $x$ and $y$ are $(n,\eps)$-separated if 
$\exists\, 0\le k\le n-1,\ d(f^{k}(x),f^{k}(y))\ge\eps.$  A $(n,\eps)$-ball at $x$ is the set of all points that are $(n,\eps)$-close to $x$. A finite subset $\Gamma$ of $\Lambda$ is called
 $(n,\eps)$-spanning if any point of $\Lambda$ is contained in a $(n,\varepsilon)$-ball of some point in $\Gamma$.

 J. Yang (see \cite[Theorem A]{Yan14})  proved that the time-one map $f$ is entropy expansive on the singular hyperbolic attractor $\Lambda$. This yields the upper semi-continuity for the metric entropy in the weak star topology and it is well known that it implies:

\begin{lemma}\label{Lem:usc}
Every continuous function $\psi$ on $\Lambda$ has an equilibrium state $\mu$ supported on $\Lambda$.
\end{lemma}

Denote by $J^{cu}$ the Jacobian for $f$ restricted to the $F^{cu}$ direction, \ie $J^{cu}=\det {\rm D}f_{|F^{cu}}$. We also set  $V:=-\log |J^{cu}|$. It is a continuous function. Thus, by Lemma~\ref{Lem:usc}, there exists at least one equilibrium state for $V$. 
Our goal is to show that this equilibrium state is unique and is the SRB measure.



%
%

%

\subsection{Plan of the proof}
The paper proceeds as follows. 
In Section \ref{sec-existence} we prove ${\CP(V)=0}$.
This shows that any equilibrium state for $ V=-\log \left|J^{cu}\right|$ satisfies the Pesin entropy formula, and therefore, such an equilibrium state  is an SRB measure  if and only if it admits a positive Lyapunov exponent (by \cite{LeY85}). Then, we prove that an equilibrium state for $V$ has one positive Lyapunov exponent.

In Section \ref{sec-uniqu} we prove  that any SRB measure is physical  and then deduce the uniqueness from  transitivity. 

In Appendix \ref{appen} we recall several classical properties and state precisely some  definitions such as those of strong stable manifolds, the Oseledets splitting of invariant measures, the absolutely continuity, etc. 
 
\begin{Acknowledgements}

{We thank R. Scott for correcting english,  J. Buzzi, Y. Cao, S. Crovisier, Y. Hua, F. Ledrappier  and the anonymous referee for useful discussions, advices and improvements  for the redaction of this paper.}

\end{Acknowledgements}


\section{The Existence of ergodic SRB measures}\label{sec-existence}

In this section we prove that ${\CP(V)=0}$. As we said above, a standard argument will imply that any equilibrium state for $V$, is an SRB measure. 

\subsection{Upper bound for ${\CP(V)}$}

\begin{lemma}\label{lem-presbounded}
{$\CP(V)\le 0$.} Moreover, if $\CP(V)=0$ then, there exists an ergodic measure which satisfies the Pesin entropy formula and is  an equilibrium state for $V$.
\end{lemma}
\begin{proof} We remind the reader that $V(x)=-\log |J^{cu}(x)|$. We prove the lemma by the following steps. {Let $\mu$ be an ergodic invariant probability measure.}

$\bullet$ {If $\mu$ is concentrated on a singularity, by the definition of sectional expansion, we have $\int V\,d\mu<0$ and $h_{\mu}(f)=0$.} Then 
$$h_\mu(\varphi_1)+\int V d\mu< 0.$$

$\bullet$ If $\mu$ is not concentrated on a singularity, by Lemma \ref{Lem:measure-hyperbolic} , we have the following measurable Oseledets splitting 
$$E^{ss}\oplus \underbrace{\left<X\right>\oplus E^{u}_\mu}_{=F^{cu}}.$$
Note that $\disp -\int V\,d\mu$ is the integral of the sum of { all Lyapunov exponents along $F^{cu}$. Among them, the Lyapunov exponent along the flow direction is zero and the others are positive. Therefore, the sum of positive Lyapunov exponents equals the sum of Lyapunov exponents along $F^{cu}$; \ie} 
$$\int V\,d\mu=-\sum \la^{+}_{\mu},$$
where $\sum \la^+_{\mu}$ is the sum of all positive Lyapunov exponents of $\mu$. By the {Ruelle inequality},  we have
$$h_{\mu}(f)\le \sum{\la^+_{\mu}}=-\int V\,d\mu.$$
$\bullet$ Now we thus have
$$\CP(V)\le \sup_{\mu~{\rm ergodic}}\{h_\mu(\varphi_1)+\int V d\mu\}\le 0.$$

The second statement in the Lemma now easily follows:  if $\CP(V)=0$, let { $\mu$ be an equilibrium state for $V$ (it exists due to Lemma \ref{Lem:usc}). Therefore  $\CP(V)=h_{\mu}(f)+\int V\,d\mu=0$ holds. Hence, $h_{\nu}(f)+\int V\,d\nu=0=\CP(V)$ holds for at least one ergodic component $\nu$ of $\mu$ (see \cite[Th. 4.1.12 and Cor. 4.3.17]{HK}).  
Thus, $\nu$ is ergodic, satisfies the Pesin entropy formula and is an equilibrium state for $V$.}
\end{proof}

\subsection{Lower bound for {$\CP(V)$}}

We follow and adapt the volume lemma from Bowen (see \cite[Section 4.B]{Bowen-lnm470}) and its version from Qiu (see \cite{Qiu11}). 

\begin{proposition}
\label{prop-lowerboundpress}
Let $\La$ be the singular hyperbolic attractor with splitting $T_{\La}M=E^{ss}\oplus F^{cu}$. Then 
$\CP(V)\ge 0$. 
\end{proposition}


\begin{proof}
 We use the characterization that we recalled just after the definition for pressure:
$$\CP(V)=\lim_{\eps\to0}\limsup_{\ninf}\frac1n\inf\left\{\sum_{x\in \CE_{n,\eps}}e^{S_{n}(V)(x)},\ \CE_{n,\eps}=\textrm{minimal~}(n,\eps)\text{-spanning set in }\La\right\}.$$

Let $\eps>0$ be fixed and consider a minimal $(n,\eps)$-spanning set $\CE_{n,\eps}=\{\xi_{i}\}$. Therefore $\disp \bigcup B_{n}(\xi_{i},\eps)$ is an open neighborhood of $\Lambda$.

\begin{lemma}
There is $\varepsilon_0>0$ such that for any $\varepsilon\in(0,\varepsilon_0)$, there is $m_{\eps}>0$ {independent of $n$} such that 
any cover of $\Lambda$  by $(n,\eps)$-balls has Lebesgue measure  larger than $m_\varepsilon$.
\end{lemma}
\begin{proof}
{ First, we claim that  $\Lambda$ contains regular orbits. Indeed, if this does not hold, then $\Lambda$  only consists on singularities and transitivity yields that $\Lambda$ is a single singularity, say $\s$.  The bundle $F^{cu}$ is non-trivial and sectional expanding. Therefore the unstable space $E^{u}(\sigma)$ is non-trivial. Consequently, $\Lambda=\{\sigma\}$ cannot be an attractor. }


By \cite[Theorem A]{Yan14}, the topological entropy on $\Lambda$ is  { thus positive.} 
By the variational principle, there is an $f$-invariant ergodic measure $\mu$   { with} positive metric entropy. Since $\mu$ is non-trivial and supported on a singular hyperbolic attractor, $\mu$ is hyperbolic as stated in Lemma \ref{Lem:measure-hyperbolic}. We can use Pesin theory for that measure and construct local unstable manifolds $W^{u}_{loc}(x)$ for $\mu$-almost every point $x$. 
 
 One can choose a point $x\in\Lambda$ such that { $W^{u}_{loc}(x)$ contains a disk for the topology of the embedded manifold $W^{u}_{loc}(x)$, and whose diameter for the metric of $M$ is $\varepsilon_1>0$.} We denote this disk by $D^{u}(\eps_1)$. Since $\Lambda$ is a topological attractor, the unstable manifold of any point of $\Lambda$ is contained in $\Lambda$. Thus, $D^{cu}(\eps_1):=\disp \bigcup_{t\in[-\eps_1,\eps_1]}\varphi_{t}(D^{u}(\eps_1))$ is contained  in $\Lambda$. This set has a Lebesgue measure of order $\eps_1^{\dim F^{cu}}$. See more in Subsec.~\ref{sssec-absconti} about the metrics between the manifold $M$ and sub-manifolds.

Take $\eps$ is smaller than $\eps_1$ and $\eps_{\Lambda}$ (which is the size of local strong stable manifolds from  Subsection~\ref{ssec-localmanifold}). All the points in $D^{cu}(\eps_1)$ admit a stable local manifold $W^{ss}_{\eps}$. Then, $\disp \bigcup B_{n}(\xi_{i},\eps)$ contains  the set 
 $$\bigcup_{y\in D^{cu}(\eps_1)} W^{ss}_{\eps/4}(y),$$ 
 which has a Lebesgue measure of order $\eps^{\dim E^{ss}}$ (using the absolute continuity of the stable foliation, see Subsec. \ref{sssec-absconti}).   
 \end{proof}

%

%

We assume that $\eps$ is much smaller than this $\eps_{0}$, where $\eps$ is used to define minimal $(n,\eps)$-spanning sets. Since $\Lambda$ is an attractor, the bundle $E^{ss}$ and the local strong stable foliation $W^{ss}_{loc}$ can be extended in a neighborhood of $\Lambda$ (see \cite[Theorem 5.5]{HPS77} and Subsec. \ref{ssec-localmanifold}).  

Define $B^{cu}(x,\rho)$ as the balls of center $x\in\Lambda$ and radius $\rho<\eps_{0}$ respectively in the local $cu$-plaque $W^{cu}_{loc}(x)$.  Denote by
$$B^{cu}_n(x,\rho)=\{y:~f^i(y)\in B^{cu}(f^i(x),\rho),~\forall 0\le i\le n-1\}.$$

The metrics between $M$ and local strong stable manifolds have uniform ratios ( see the property (D1-D3) in Subsec. \ref{sssec-absconti}). More precisely, there is $\kappa>1$ such that for all $\varepsilon$ small enough, 
\begin{equation}
W^{ss}_{\rho/\kappa}(B^{cu}(x,\rho/\kappa))\subset B(x,\rho)\subset W^{ss}_{\kappa\rho}(B^{cu}(x,\kappa\rho)).
\end{equation}

The continuity {of} $Df$ and the continuity of $F^{cu}$ imply that there exists $C_{\eps}>0$ such that for every $x$ and $y$ in $\La$, 
\begin{equation}
\label{eq1-disto1shoot}
d(x,y)<\kappa\eps\Longrightarrow e^{-C_{\eps}}\le \left|\frac{J^{cu}(x)}{J^{cu}(y)}\right|\le e^{C_{\eps}}.
\end{equation}
{Note that by continuity of $F^{cu}$, $C_\eps\to 0$ as $\eps\to 0$.}

{Again, the absolute continuity of the stable foliation yields:}
 \begin{eqnarray*}
m_{\eps}&\le &\sum_{i}\Leb(B_{n}(\xi_{i},\eps))\le \sum_{i}\Leb\left(W^{ss}_{\kappa\varepsilon}(B_n^{cu}(\xi_i,\kappa\eps))\right)\\
&\le& C_{s}(\kappa.\eps)^{\dim E^{ss}}.\sum_{i}\int_{f^{n-1}(B_{n}^{cu}(\xi_{i},\kappa\eps))}e^{S_{n}(V)(x)}dx\\
&\le&  C_{s}(\kappa.\eps)^{\dim E^{ss}}.e^{nC_{\eps}}\sum_{i}e^{S_{n}(V)(\xi_{i})}\Leb^{cu}(B^{cu}(f^{n-1}(\xi_{i}),\kappa.\eps))\\
&\le&  C_{s}C_{cu}(\kappa.\eps)^{\dim E^{ss}}.e^{nC_{\eps}}\sum_{i}e^{S_{n}(V)(\xi_{i})}.
\end{eqnarray*}
Where  the second inequality (thus the constant $C_s$) comes from \eqref{equ-controllebprod} (see Subsec.~\ref{sssec-absconti}) and conformity of the Lebesgue measure. $C_{cu}$ is a uniform upper bound of the Lebesgue measure on $B^{cu}(x,\rho)$.

Taking $\disp\frac1n\log$ in the last inequality, then taking $\limsup_{\ninf}$ and then $\eps\to0$ we get
 $$\CP(V)\ge 0.$$
\end{proof}

\subsection{Proof for the existence of ergodic SRB measures}

The map $f$ is $\CC^{2}$, we can thus use the Pesin theory. { By Lemma~\ref{Lem:usc},  $V$ has an equilibrium state. By Proposition \ref{prop-lowerboundpress}, $\CP(V)\ge 0$ and by Lemma \ref{lem-presbounded}, $\CP(V)\le 0$. This yields $\CP(V)=0$. Then, the second part of Lemma \ref{lem-presbounded} yields the existence of  an ergodic equilibrium state for $V$, $\mu$,  which satisfies:}
$$h_{\mu}(f)=\sum\la^{+}_{\mu}.$$
The fact that $F^{cu}$ is sectional expanded and  has dimension at least 2 implies that every ergodic measure has a positive Lyapunov exponent. So does $\mu$. 


Therefore, $\mu$ is an SRB measure and the proof of the existence of ergodic SRB measures in Theorem \ref{th-main} is complete.


\section{The SRB measure is unique and physical}\label{sec-uniqu}

Almost all ergodic components of an SRB measure are also SRB measures (see \cite[Corollaire 4.10]{ledrappier-srb}). To get uniqueness for the SRB measure, it is thus sufficient to get uniqueness for ergodic SRB measures. The uniqueness will follow from the fact that the disintegration of an SRB measure for a $C^{1+\al}$ diffeomorphism is more than absolutely continuous with respect to $\Leb^{u}$; it is actually equivalent to it.

\begin{lemma}
\label{lem-muuequivleb}
Let $\mu$ be an ergodic SRB measure for $f=\varphi_1$. Let $\{\mu_{u}\}$ be its system of conditional measures with respect to any measurable partition subordinate to the unstable foliation. Then almost every $\mu_{u}$ is equivalent to $\Leb^{u}$. 
\end{lemma}
\begin{proof}
Remember (see Lemma \ref{Lem:measure-hyperbolic}) that there exist two measurable subbundles $E^{c}=\left<X\right>$ and $E^{u}_\mu$ defined $\mu$-almost everywhere  and invariant {for the time-one map $f$} such that 
$$E^{ss}\oplus \underbrace{\left<X\right>\oplus E^{u}_\mu}_{=F^{cu}}.$$
Let $\xi$ be a partition subordinate to the {(Pesin)} unstable foliation $W^{u}$ {of the time-one map $f$}. {From the remark in \cite[Page 513]{LeY85}}, for $\mu$-a.e. $x$, the density $\rho_{x}$ of $\mu_{u}$ on $\xi(x)$ with respect to $\Leb^{u}$ satisfies
$$\frac{\rho_{x}(z)}{\rho_{x}(y)}=\prod_{k=0}^{+\8}\frac{\left|\det Df_{|E^{u}_\mu}(f^{-k}(z))\right|}{\left|\det Df_{|E^{u}_\mu}(f^{-k}(y))\right|}.$$
By the (non-uniformly) backward contraction property of the unstable manifold, the above quantity converges and is also bounded away from zero,  uniformly in $y$ and $z$ chosen in $\xi(x)$. This shows that the density is a continuous function on $\xi(x)$ and  does not vanish. 
\end{proof}

\bigskip
Let us now prove the uniqueness for the SRB measure and that it is physical. {For the map $f$ and the measure $\mu$, a point $x$ is said to be \emph{forward $\mu$-generic} if $1/n\sum_{i=0}^{n-1}\delta_{f^i(x)}\to\mu$ as $n\to\infty$.}

\begin{lemma}
\label{lem-srb-phys}
Let $\mu$ be an ergodic SRB measure. Then there exists {a non-empty open} set $U$ such that Lebesgue almost every point in $U$ is forward $\mu$-generic. 
\end{lemma}
\begin{proof}
Since $\mu$ is an SRB measure, we have that $\disp h_{\mu}+\int V\,d\mu=0$. Thus $\mu$ cannot be concentrated on a singularity (see the first part of the proof of Lemma \ref{lem-presbounded}). Then it is hyperbolic for the flow $\varphi_t$ by Lemma~\ref{Lem:measure-hyperbolic}. Let $\xi$ be any measurable partition subordinate to the unstable foliation (for $\mu$) and choose $x_{1}$ forward $\mu$-generic. 
 We can assume that $x_{1}$ is in the interior of $\xi(x_{1})$. By Lemma \ref{lem-muuequivleb}, $\Leb^{u}$-a.e. every $x$ in $\xi(x_{1})$ is forward $\mu$-generic.  
Clearly, if $x$ is {forward} $\mu$-generic, then for every $t$, $\varphi_{t}(x)$ is also forward $\mu$-generic. Moreover, every $y$ in $W^{ss}(x)$ is also {forward} $\mu$-generic. Let us set
$$U:=\bigcup_{t\in]-\eps,\eps[,\ x\in \xi(x_{1})} W^{ss}_{\eps}(\varphi_{t}(x)).$$
{By the hyperbolicity of $\mu$, $\xi_1$ is a $(\dim F^{cu}-1)$-dimensional sub-manifold. Consequently, $\varphi_{[-\eps,\eps]}(\xi(x_1))$ is a $\dim F^{cu}$-dimensional sub-manifold.} Hence $U$ is an open set by the invariance of domain theorem. {By the absolute continuity of the strong stable foliation,} Lebesgue almost every point in $U$ is forward $\mu$-generic. 
\end{proof}

\begin{remark}
With almost the same proof of Lemma~\ref{lem-srb-phys}, we have the following result. If $\mu$ is a hyperbolic SRB measure for a $C^2$ diffeomorphism $f$ or a $C^2$ vector field $X$, if ${\rm supp}(\mu)$ admits a partially hyperbolic splitting $E^{ss}\oplus F$ satisfying that $\dim E^{ss}$ is exactly the sum of the multiplicities of negative Lyapunov exponents of $\mu$, then there is a non-empty open set $U$ in the manifold such that Lebesgue almost every point in $U$ is forward $\mu$-generic.

\end{remark}

We can now prove the uniqueness of the SRB measure. 
Assume that $\mu_{1}$ and $\mu_{2}$ are two different ergodic SRB measures. Let $U_{1}$ and $U_{2}$ obtained from Lemma \ref{lem-srb-phys} for $\mu_{1}$ and $\mu_{2}$. 
Now, transitivity shows that for some $T$, $\varphi_T(U_{2})\cap U_{1}$ is a non-empty open set. Lebesgue almost every point in this intersection is {forward} $\mu_{1}$ and $\mu_{2}$-generic, which shows that $\mu_{1}=\mu_{2}$.

\appendix
\section{Classical results on (partially) hyperbolicity}\label{appen}
\subsection{Oseledets splitting for ergodic measures}
For an ergodic invariant measure $\mu$, denote by $E_\mu^s$ the sub-bundle associated to all the negative Lyapunov exponents and $E_\mu^u$ the sub-bundle associated to all the positive Lyapunov exponents.

An important property for singular hyperbolic attractors is that every invariant ergodic measure is {\it hyperbolic} in the following sense:
\begin{lemma}\label{Lem:measure-hyperbolic}
If an ergodic measure $\mu$ is not supported {on the set of singularities}\footnote{Note that $\mu$ is ergodic and it is not supported on the set of singularities. {Hence no singularity can be regular for $\mu$.}}, then it is hyperbolic, i.e., ${\rm D}\varphi_t$ has exactly one zero Lyapunov exponent that is given by the vector field $X$. Moreover,  $E^{s}_{\mu}=E^{ss}$, and $F^{cu}=\left<X\right>\oplus E^{u}_{\mu}$.
\end{lemma}
\begin{proof}[Sketch of the proof]
Since $\mu$ is ergodic and is not supported on the set of singularities, we have that almost every point $x$ of $\mu$ is a non-singular point and recurrent. Thus, there is a sequence $(t_n)$ such that $\lim_{n\to\infty}t_n=\infty$ and $\lim_{n\to\infty}\varphi_{t_n}(x)=x$. Consequently $\lim_{n\to\infty}X(\varphi_{t_n}(x))=X(x)\neq 0$. The Lyapunov exponent along $X(x)$ is
$$\lim_{t\to\infty}\frac{1}{t}\log\frac{|D\varphi_t X(x)|}{|X(x)|}=\lim_{n\to\infty}\frac{1}{t_n}\log\frac{|D\varphi_{t_n} X(x)|}{|X(x)|}=\lim_{n\to\infty}\frac{1}{t_n}\log\frac{|X(\varphi_{t_n}(x))|}{|X(x)|}=0.$$

Now we show that there is at most one zero Lyapunov exponent. Note that the non-zero vector that has zero Lyapunov exponent cannot be contained in $E^{ss}$. Thus one can assume that it is contained in $F^{cu}$. {Assume by contradiction} that there is a vector $v\in F^{cu}(x)\setminus\left<X(x)\right>$, such that the Lyapunov exponent of $v$ is zero. We consider the plane generated by $X(x)$ and $v$. The sectional expansion property implies that the area of this plane is expanded by $D\varphi_t$. However, this contradicts to the fact that the two vectors $X(x)$ and $v$ all have zero Lyapunov exponents.

The above arguments also implies that $E^{s}_{\mu}=E^{ss}$, and $F^{cu}=\left<X\right>\oplus E^{u}_{\mu}$.
%
%
\end{proof}

\subsection{Local stable manifolds and central unstable plaques}\label{ssec-localmanifold}

\subsubsection{Local stable manifolds}
One can continuously extend the bundles $E^{ss}$ and $F^{cu}$ in a neighborhood  $U$ of $\Lambda$ (see \cite[Proposition 2.7]{CrP15}). Note that the extension of  $E^{ss}$ can be made $Df$-invariant (see {the proof of} \cite[Corollary 2.8]{CrP15}\footnote{The statment of \cite[Corollary 2.8]{CrP15} does not give this statement directly. However, its proof applies. More precisely, we do not use the perturbation part of diffeomorphisms, we only use the part of the neighborhood; from the proof, we know that the dominated bundle (the weaker bundle) can be extended in a neighborhood in an invariant way for an attractor.}, but generally {we do not know how to} extend $F^{cu}$ to be an invariant bundle in $U$ . 
We still denote the extended bundles by $E^{ss}$ and $F^{cu}$.




Then, since $\Lambda$ is an attractor, one can construct a stable foliation in $U$. 
More precisely, there is  $\varepsilon_\Lambda>0$ such that for any $x\in U$, $W^{ss}_{\varepsilon_\Lambda}$ is a local embedded sub-manifold of dimension $\dim E^{ss}$ by {\cite[Theorem 4.3]{CrP15} and the proof of} \cite[Theorem 4.1]{HPS77}.
There are $C^1$ maps $\Psi^{ss}_x:~E^{ss}(x)\to F^{cu}(x)$ that are $C^1$ close to the zero maps (on a compact neighborhood of the zero section) uniformly w.r.t. $x$ and vary continuously w.r.t. $x$ in the $C^1$ topology such that 
$$W^{ss}_{\eps_{\Lambda}}(x)=\exp_x({\rm Graph}(\Psi_x^{ss})(\varepsilon_\Lambda))$$
where ${\rm Graph}(\Psi_x^{ss})(\varepsilon_\Lambda)=\{(v,\Psi_x^{ss}(v)):~v\in E^{ss}(x),~\|v\|\le\varepsilon_\Lambda\}$. Note that $C$ and $\lambda$ are the constants in the definition of the singular hyperbolicity. One can choose $\varepsilon_\Lambda$ such that every point $y\in W^{ss}_{\eps_\Lambda}$ satisfies:
$$d(\varphi_t(x),\varphi_t(y))\le 2C{\rm e}^{-\lambda t/2}, \textrm{and}~\frac{d(\varphi_t(x),\varphi_t(y))}{\|D\varphi_t|_{F(x)}\|}\le  2C{\rm e}^{-\lambda t/2}, \textrm{~~~~~~~} \forall t\ge 0.$$

More generally, one can define $W^{ss}_{\eps}(x)$ for every $\eps<\eps_{\Lambda}$ by taking the restricted graph.

These (local) strong stable manifolds generate a (global) stable foliation in a neighborhood of $\Lambda$, 
$$W^{ss}(x):=\bigcup_{t\ge 0} \varphi_{-t}W^{ss}_{\eps_{\Lambda}}(\varphi_{t}(x)).$$

\subsubsection{$cu$-plaques}
These are {locally invariant center-unstable sub-manifolds (called plaques in \cite[Theorem 5.5]{HPS77}, {or \cite[Theorem 4.5]{CrP15}}}) $W^{cu}_{loc}(x)$ tangent to $F^{cu}(x)$ and containing the image by {$\exp_{x}$} of a disk of radius $\eps_{\Lambda}$. 

More precisely,
for each point $x$, there are $C^1$ maps $\Psi^{cu}_x:~F^{cu}(x)\to E^{ss}(x)$ that are $C^1$ close to the zero maps (on a compact neighborhood of the zero section) uniformly w.r.t. $x$ and vary continuously w.r.t. $x$ in the $C^1$ topology such that
$$W^{cu}_{\eps_{\Lambda}}(x)=\exp_x({\rm Graph}(\Psi_x^{cu})(\varepsilon_\Lambda)),$$
where ${\rm Graph}(\Psi_x^{cu})(\varepsilon_\Lambda)=\{(v,\Psi_x^{cu}(v)):~v\in F^{cu}(x),~\|v\|\le\varepsilon_\Lambda\}$.

\subsubsection{Absolute continuity, local product structure and metrics for sub-manifolds}\label{sssec-absconti}
The stable foliation has the \emph{absolutely continuous} property. Given $y\in W^{ss}_{\varepsilon_\Lambda}(x)$, any two {transversals} $\Sigma_y$ at $y$ and $\Sigma_x$ at $x$ to the stable foliation, the stable foliation induces a holonomy map $h^{\Sigma_x,\Sigma_y}$. When the vector field is $C^2$, $(h^{\Sigma_x,\Sigma_y})_*{\rm Leb}_{\Sigma_x}$ is equivalent to ${\rm Leb}_{\Sigma_y}$.
Furthermore, the density function can be uniformly bounded when $\Sigma_x$ and $\Sigma_y$ are tangent to the cone field ${\mathcal C}_a^F$ for $a$ small, where at each point $x$
$${\mathcal C}_a^F=\{v\in T_x M:~v=v^s+v^c,~v^s\in E^{ss}(x),~v^c\in F^{cu}(x),~|v^s|\le a|v^c|\}.$$ 
One can see \cite[Theorem 7.1]{Pes04} for the details of the absolute continuity of the stable foliation. 
Then, the absolute continuity yields for any measurable set $A$ in a $cu$-plaque $W^{cu}_{loc}(x)$,
\begin{equation}\label{equ-controllebprod}
\frac{1}{C_{s}} \eps^{\dim E^{ss}}Leb^{cu}_{W^{cu}_{loc}(x)}(A)\le Leb(W^{ss}_{\kappa\eps}(A))\le C_{s} \eps^{\dim E^{ss}}Leb^{cu}_{W^{cu}_{loc}(x)}(A).
\end{equation} 

Roughly speaking, local stable manifolds and $cu$-plaques generate a local product structure with 
$$B(x,\eps)\approx W^{ss}_{\eps}(x)\times W^{cu}_{\eps}(x),$$
and absolutely continuity implies  a local equivalence $Leb\sim Leb^{ss}_{x}\otimes Leb^{cu}_{x}$.  

There exists some universal constant $\kappa>1$ {independent of $x$ and $\varepsilon$} such that for every small enough $\eps>0$, and for every $x$, 
\begin{enumerate}
\item[(D1)] $\disp W^{ss}_{\eps/\kappa}(B^{cu}(x,\eps/\kappa))\subset B(x,\eps)\subset W^{ss}_{\kappa\eps}(B^{cu}(x,\kappa\eps))$,
\item[(D2)] for any $y,z\in W^{ss}_\varepsilon(x)$, we have $d_{M}(y,z)/\kappa\le d^{ss}(y,z)\le \kappa d_{M}(y,z)$, where $d^{ss}$ is the metric on $W^{ss}_{\varepsilon_\Lambda}(x)$ and $d_M$ is the metric on $M$,
\item[(D3)] for any $y,z\in W^{cu}_\varepsilon(x)$, we have $d_{M}(y,z)/\kappa\le d^{cu}(y,z)\le \kappa d_{M}(y,z)$, where $d^{cu}$ is the metric on $W^{cu}_{\varepsilon_\Lambda}(x)$ and $d_M$ is the metric on $M$. 
\end{enumerate}
These properties are referred as the uniform control between the metric of $M$ and the metrics of its sub-manifolds.



R. Leplaideur\\
LMBA, UMR6205\\
Universit\'e de Brest\\
6, avenue Victor Le Gorgeu\\
C.S. 93837, France \\
\texttt{Renaud.Leplaideur@univ-brest.fr}\\
\texttt{http://pagesperso.univ-brest.fr/$\sim~$leplaide/}

D. Yang\\
School of Mathematical Sciences, Soochow University,\\
No. 1, Shizi Street, Suzhou, 215006, P.R. China\\
yangdw@suda.edu.cn

\end{document}